\documentclass[12pt]{article}
\usepackage{amsthm}
\usepackage{amsfonts}
\usepackage{amssymb}
\usepackage{amsmath}
\usepackage{nccmath}
\usepackage{mathtools}
\usepackage{mathrsfs}
\usepackage{setspace}
\usepackage{graphicx}
\usepackage[export]{adjustbox}
\usepackage{hyperref}
\usepackage[numbers,sort&compress]{natbib}
\usepackage{enumerate}
\usepackage{authblk}
\usepackage{placeins}

\newtheorem{theorem}{Theorem}[section]
\newtheorem{lemma}[theorem]{Lemma}

\theoremstyle{definition}
\newtheorem{definition}[theorem]{Definition}
\newtheorem{example}[theorem]{Example}

\newtheorem{corollary}[theorem]{Corollary}

\theoremstyle{remark}
\newtheorem{remark}[theorem]{Remark}

\numberwithin{equation}{section}

\usepackage{amsmath,amssymb,graphicx} %load extra symbols and environments
\usepackage[margin=1in]{geometry} %set margins
\usepackage{enumerate}
\usepackage{authblk}
\usepackage{placeins}
\usepackage{array}
\usepackage{academicons}
\usepackage{xcolor}
\usepackage{svg}
\usepackage[utf8]{inputenc}
\usepackage{booktabs}
\usepackage{tikz}
\usepackage{placeins}
\usepackage{balance}
\usepackage{longtable}
\DeclareMathOperator{\shift}{E}

\providecommand{\keywords}[1]{\textbf{\textit{Keywords:}} #1}
\providecommand{\subjclass}[1]{\textbf{\textit{MSC2020:}} #1}
\begin{document}

\nocite{*} % this command forces all references in template.bib to be printed in the bibliography
%\maketitle

\title{Equality of the Casoratian and the Wronsikian for some Sets  of Functions and an Application in Difference Equations}

\author{Hailu Bikila Yadeta \\ email: \href{mailto:haybik@gmail.com}{haybik@gmail.com} }
  \affil{Salale University, College of Natural Sciences, Department of Mathematics\\ Fiche, Oromia, Ethiopia}
\date{\today}
\maketitle
%\item \href{https://orcid.org/0000-0002-1232-2096}{\textcolor{orcidlogocol}{\aiOrcid} \hspace{2mm} %orcid.org/0000-0002-1232-2096}
\noindent
\begin{abstract}
\noindent The Casoratian determinants are very important in the study of linear difference equation, just as the Wronskian determinants are very important in the study of linear ordinary differential equations. The Casoratian and Wronskian determinants of a given set of functions are generally different. In this paper, we show that the Wroskian $ W[1, x, x^2,...,x^n ]$  and the Casoratian $ C[1,x, x^2,...,x^n ] $ of the set of functions $ \{1, x, x^2,...,x^n: \,  n \in \mathbb{N} \}$ are equal to each other and independent of the variable $x$. Furthermore, we show that $  W[p_1, p_2,...,p_n, p_{n+1} ]= C[p_1, p_2,..., p_n, p_{n+1} ]$ for any basis $ \mathcal{B} = \{  p_1, p_2,..., p_n, p_{n+1}    \}$ of the vector space spanned by $ \{ 1, x, x^2,...,x^n \}$. We also explore applications of such Casoratians in deriving the general solutions of certain classes of homogeneous linear difference equations.Additionally, the paper discusses scenarios where the Casoratian and Wronskian are proportional, as well as cases where they differ for finite sets of functions.
\end{abstract}

\noindent\keywords{ Wronskian, Casoratian, linear dependence, span, difference equations, shift operators, periodic functions }\\
\\

\section{Introduction}

\begin{definition}
  Let $f_1, f_2,f_3,...,f_n $ are $(n-1)$ times differentiable real-valued or complex-valued functions  defined on some open interval $I$. Then the Wronskian of the set of functions is denoted  by $ W(f_1,f_2,f_3,...,f_n)(x)$  is the function defined on $I$ by the  determinant
  \begin{equation}\label{eq:Wronskian}
W(f_1,f_2,f_3,...,f_n)(x)=  \begin{vmatrix}
 f_1(x)   &   f_2(x)       & f_3(x)         & \cdots    & f_n(x)\\
f_1^{'}(x)&   f_2^{'}(x)       & f_3^{'}(x)      & \cdots    & f_n^{'}(x)\\
f_1^{''}(x)&   f_2^{''}(x)       & f_3^{''}(x)      & \cdots    & f_n^{''}(x)\\
\vdots  & \vdots    & \vdots    & \ddots    & \vdots\\
f_1^{n-1}(x)&   f_2^{n-1}(x)       & f_3^{n-1}(x)      & \cdots    & f_n^{n-1}(x)
  \end{vmatrix}.
\end{equation}
\end{definition}

\begin{definition}
  Let $f_1, f_2,f_3,...,f_n $ are  real-valued or complex-valued functions defined on  $\mathbb{R}$. Then the Casoratian of the functions, denoted  by $ C(f_1,f_2,f_3,...,f_n)(x) $  is the function defined by the  determinant
  \begin{equation}\label{eq:Casoratian}
C(f_1,f_2,f_3,...,f_n)(x)=  \begin{vmatrix}
 f_1(x)   &   f_2(x)       & f_3(x)         & \cdots    & f_n(x)\\
f_1(x+1)&   f_2(x+1)       & f_3(x+1)      & \cdots    & f_n(x+1)\\
f_1(x+2)&   f_2(x+2)       & f_3(x+2)      & \cdots    & f_n(x+2)\\
\vdots  & \vdots    & \vdots    & \ddots    & \vdots\\
f_1(x+n-1)&   f_2(x+n-1)       & f_3(x+n-1)      & \cdots    & f_n(x+n-1)
  \end{vmatrix}.
 \end{equation}
\end{definition}
The Casoratian plays a role in the study of linear difference equations similar to that played by the Wronskian for linear differential equations. In this paper we prove the equality of the Casorantian  and Wronskian of the set $\{1,x,x^2,...,x^n, \, n \in \mathbb{N} \}$. The value for the Casoratian is less obvious while that of Wronskian is more straightforward. We also give the application of such Casoratian in the  general solution of some class of homogeneous linear difference equations.
\begin{definition}
Consider the homogeneous linear difference equation
\begin{equation}\label{eq:lineardifferenceequation}
  a_n(x)y(x+n)+a_{n-1}y(x+n-1)+...+a_1(x)y(x+1)+a_0(x)y(x)=0,
\end{equation}
where $a_i(x), i=0,1,...n $ are defined on some closed interval $[a,b]$ and $a_n(x)a_0(x) \neq 0 $ for all $x\in [a,b] $.  A of solutions $ y_1, y_2,...y_n $ is termed as \emph{fundamental set of solutions } if  the Casoratian $ C(y_1,y_2,y_3,...,y_n)(x) \neq 0  $ for all $ x\in [a,b] $.
\end{definition}
Fundamental set of solutions is important in determining the general solutions of difference equations of the form (\ref{eq:lineardifferenceequation}). Further discussion of difference equations, fundamental set of solutions, etc.  are available in several literatures. See, for example,\cite{KP},\cite{CHR}, \cite{GR}, \cite{KM}, \cite{LB}. The Casoratian and the Wronskian  of a given set of functions are generally different. For instance,
\begin{example}
  Consider the set $\{\cos x, \sin x \}$.
  $$ W[\cos x, \sin x]=1  \neq  \sin 1 = C[\cos x, \sin x] . $$
\end{example}

\section{A case of equality of the Wronskian and  the Casoratian }
\begin{theorem}
   The Wronskian of the set of functions $ \{1, x, x^2,...,x^n \}$,  is independent of the variable $x$, is the number $\prod_{k=0}^{n}k! $.
\end{theorem}
\begin{proof}
By definition of Wronskian, we have
\begin{equation}\label{eq:WronskianofPowers}
 W[1, x, x^2,...,x^n ] =
\begin{vmatrix}
 1   &   x       & x^2         & \cdots    & x^n \\
0 &   1      & 2x      & \cdots    & nx ^{n-1}\\
0  &   0       & 2      & \cdots    & n(n-1)x^{n-2}\\
\vdots  & \vdots    & \vdots    & \ddots    & \vdots\\
0 &   0       & 0      & \cdots    & n!
  \end{vmatrix}.
\end{equation}
 The result easily follows from the fact that the determinant of the upper triangular matrix is the product of all its diagonal entries.
\end{proof}

\begin{remark}
  The number
  \begin{equation}\label{eq:superfactorialn}
   \prod_{k=0}^{n}k! = 0! \times 1!\times  2! \times ...\times n!:=sf(n)
  \end{equation}
  is termed as the \textit{superfactorial} of $n$.
\end{remark}
\begin{lemma}
The Vandermonde determinant
\begin{equation}\label{eq:Vandermonde}
\begin{vmatrix}
1 & x_{0} & x_{0}^{2} & \dots & x_{0}^{n} \\
1 & x_{1} & x_{1}^{2} & \dots & x_{1}^{n} \\
\hdotsfor{5} \\
1 & x_{n} & x_{n}^{2} & \dots & x_{n}^{n}
\end{vmatrix}
=
\prod_{0\leq j < i \leq n} (x_{i} - x_{j})
\end{equation}
\end{lemma}

\begin{theorem}
  The Casoratian of the set of function $ \{1, x, x^2,...,x^n \},$ is independent of the variable $x$ and is equal to $\prod_{k=0}^{n}k! $.
\end{theorem}
\begin{proof}
 We want to calculate the Casorati determinant
   \begin{equation}\label{eq:CasoratianofPowers}
C(1,x, x^2,...,x^n)(x)=  \begin{vmatrix}
 1   &   x       & x^2         & \cdots    & x^n \\
1 &   x+1       & (x+1)^2     & \cdots    & (x+1)^n\\
1 &   x+2       & (x+2)^2      & \cdots    & (x+2)^n\\
\vdots  & \vdots    & \vdots    & \ddots    & \vdots\\
1 &   x+n       & (x+n)^2      & \cdots    & (x+n)^n
  \end{vmatrix}.
 \end{equation}
 The Casorati determinant (\ref{eq:CasoratianofPowers}) is a type of Vandermonde determinant (\ref{eq:Vandermonde}), where
 \begin{equation}\label{eq:xjxplusj}
  x_j= x+j,\quad j=0,1,2,...,n.
 \end{equation}
 Therefore,
 \begin{equation}\label{eq:onestepreducedCasorati}
    C(1,x, x^2,...,x^n)(x)=\prod_{0\leq j < i \leq n} (i - j).
 \end{equation}
 For $j=0$, using  $i=1,2,...,n$, we get $n!$. For $j=1$, using  $i=2,3,...,n$, we get $(n-1)!$. For $j=2$, using  $i=3,4,...,n$, we get $(n-2)!$ and so on. Consequently,
 \begin{equation}\label{eq:numericalvalueofCasorati}
   C(1,x, x^2,...,x^n)(x)=   \prod_{k=0}^{n}k!.
 \end{equation}
 \end{proof}

 \begin{theorem}
 We have the following equality.
   \begin{equation}\label{eq:equalitytheorem}
        W(1,x, x^2,...,x^n)(x)= C(1,x, x^2,...,x^n)(x) =   \prod_{k=0}^{n}k!.
 \end{equation}
\end{theorem}
 \begin{remark}
 Even if the Casoratian and  the Wronskian of a given set of functions are equal, they might be different for those of a subset of  the provided set. For instance,
    $$ W[x, x^2 ]= x^2 \neq x^2 + x = C[x, x^2 ] $$
 for the subset $\{ x, x^2\}$  of $\{1, x, x^2\}$.
 \end{remark}

 \begin{theorem}
   Let $n \in \mathbb{N} $. Let $\mathcal{P}_n$ denote the  vector space of all polynomials with degree less than or equal to $n$ and the zero polynomial. Let $\mathcal{ B}= \{ p_1,p_2,...,p_n, p_{n+1} \}$ be any basis of $\mathcal{P}_n$. Then
   \begin{equation}\label{eq:generalisedequalitytheorem}
     W[p_1,p_2 ...,p_n,p_{n+1}]= C[p_1,p_2 ...,p_n, p_{n+1}]
   \end{equation}
 \end{theorem}
 \begin{proof}
   Consider $\{1,x,x^2,...,x^n \}$ as the standard basis of $\mathcal{P}_n$. Then  each of the elements of the basis $\mathcal{B}$ can be written in terms of the standard basis
   \begin{equation}\label{eq:changeofbasis}
     p_i(x) = \sum_{j=1}^{n+1}a_{ij}x^{j-1},\quad i=1,2,...,n+1.
   \end{equation}
   Therefore
   \begin{equation}\label{eq:matrixform}
    [p_1,p_2 ...,p_n,p_{n+1}]= [1,x,x^2,...,x^n] A^t,
   \end{equation}
  where $A=(a_{ij})$ is the $n \times n $  matrix of coefficients and $A^t$ is transpose.
  \begin{align}\label{eq:transfomedequality}
  W[p_1,p_2 ...,p_n,p_{n+1}]& = W[ [1,x,x^2,...,x^n] A^t]  \nonumber \\
     & = W [1,x, x^2,...,x^n] \det A ^t   \nonumber\\
     & = C [1,x, x^2,...,x^n] \det A ^t  \nonumber\\
     &=  C[ [1,x,x^2,...,x^n] A^t]    \nonumber \\
     &=  C[p_1,p_2 ...,p_n,p_{n+1}] \nonumber \\
     & = \det A \prod_{k=0}^{n}k!
  \end{align}
 \end{proof}
 In the next theorem, we describe the equality and non equality of the Wroskian and the Casoratian a given  finite set $S \subset \mathcal{P}_n $. We assume that $ n \in N $ is the smallest integer that $S \subset \mathcal{P}_n    $.
 \begin{theorem}
   Let $ S= \{ p_0,p_1,...,p_m  \}  \subset \mathcal{P}_n $.
   \begin{itemize}
     \item  If $m \leq  n $ and $\text{span}(S)= \mathcal{P}_m  $, then
     $$ W[p_0, p_1,...,p_m]  = C[p_0, p_1,...,p_m] $$
     \item If $m \leq  n $ and $S $ is linearly dependent, then
     $$ W[p_0, p_1,...,p_m]  = 0 = C[p_0, p_1,...,p_m] $$
     \item If  $m < n $  and $\text{span}(S) \neq \mathcal{P}_k,k=1,2,3...,m   $, then
      $$ W[p_0, p_1,...,p_m]  \neq C[p_0, p_1,...,p_m] .$$
     \item If $m > n $, then $S$ is linearly dependent and
      $$ W[p_0, p_1,...,p_m] = 0 = C[p_0, p_1,...,p_m] $$
   \end{itemize}
 \end{theorem}

\subsection{Wronskian as the limit of Casoratian}
 In the next theorem we show that, if the powers of shift operators $E^i, i=1, 2 ,...,n $ in  the Casoratian determinant are replaced by the corresponding powers of the difference operate $ \triangle ^i,  i=1, 2 ,...,n $, where $\triangle := E- I $, yields same  value as the Casoratian. Linear difference equation can alternatively written interims of shift operator or difference operators. However the shift operators are used preferable for the simplicity as well as the easy determination of the order of the difference equations. See for example,\cite{KP}.

\begin{theorem}
\begin{equation}\label{eq:casoratianofdifference}
C(f_0, f_1, f_2,...,f_n)(x)=   \begin{vmatrix}
 f_0   &   f_1      & f_2         & \cdots    & f_n \\
\triangle f_0 &  \triangle f_1       & \triangle f_2     & \cdots    & \triangle f_n \\
 \triangle^2 f_0 &  \triangle^2 f_1       & \triangle^2 f_2      & \cdots    & \triangle^2 f_n \\
\vdots  & \vdots    & \vdots    & \ddots    & \vdots\\
\triangle ^n f_0 &  \triangle ^n  f_1      & \triangle ^n  f_2      & \cdots    &  \triangle ^n f_n
  \end{vmatrix}.
 \end{equation}
\end{theorem}
\begin{proof}
  In an $ n \times n $, adding a linear combination any $(n-1)$ rows to the remaining row does not change the determinant of the resulting matrix. By the linearity of the operators $\triangle^i $, the result follows:

  \begin{itemize}
    \item  Add the first row onto the second row. Equivalently $I+ \triangle =E $
  \item Add the sum of  first row plus twice the second row onto the third row. Equivalently, $ I +  2 \triangle + \triangle ^2 =E^2 $,
  \item ...............................................................................
  \item  Lastly, $ I   + n \triangle  + \binom{n}{2}  \triangle^{2}+...+\binom{n}{n-1}\triangle ^{n-1}  + \triangle ^n = E^n.   $
  \end{itemize}
  Thus the determinant of the resulting matrix is the usual Casoratian defined interims of the shift operators.
\end{proof}

\begin{lemma}
  Let $h \in \mathbb{R}, h \neq 0 $. Let $\Delta_h:= (E^h-I)$, where $E^h$ is the shift operator which is defined by $E^hy(x):=y(x+h)$, and $I$ is the identity operator. For $ n\in \mathbb{N} $, let $y$ is $n$ times continuously differentiable function on some open interval I. Then we have the following result
  \begin{equation}\label{eq:limitofdiffqutient}
    \lim_{h\rightarrow 0}\frac{\Delta_h^ny(x)}{h^n}=y^{n}(x)
  \end{equation}
  \end{lemma}

  \begin{proof}
  $ \Delta_h^n = (E^h-I)^n = \sum_{i=0}^{n} \binom{n}{r} (-1)^{n-r} E^{nr}$. By  using the fact that  the Stirling's numbers of the second kind
  $$  F(n,k):=\sum_{r=0}^{n} \frac{ (-1)^{n-r}  r^k}{r!(n-r)!} = \begin{cases}
                                                                    & 0, \mbox{ if } 0\leq k <n, \\
                                                                    & 1,  \mbox{  if  } k = n .
                                                                 \end{cases}  $$
  and  applying L'H\^{o}pital's rule to the indeterminate forms one after the other, we get
    \begin{align*}
    \lim_{h\rightarrow 0}\frac{\Delta_h^ny(x)}{h^n} & = \lim_{h\rightarrow 0}\frac{1}{h^n} \sum_{r=0}^{n} \binom{n}{r}(-1)^{n-r}y(x+hr)\\
       & =n!\lim_{h\rightarrow 0} \frac{1}{h^n} \sum_{r=0}^{n} \frac{ (-1)^{n-r} y(x+hr)}{r!(n-r)!}\\
        & =n!\lim_{h\rightarrow 0} \frac{1}{nh^{n-1}} \sum_{r=0}^{n} \frac{ (-1)^{n-r} y'(x+hr) r}{r!(n-r)!}\\
       & =.....................\\
        & = n!\lim_{h\rightarrow 0} \frac{1}{n!} \sum_{r=0}^{n} \frac{ (-1)^{n-r} y^{(n)}(x+hr) r^n}{r!(n-r)!}\\
         &= y^{(n)}(x).
    \end{align*}
  \end{proof}
\begin{theorem}
    Let the functions $y_1,y_2,...,y_n$ be $(n-1)$ times continuously differentiable on some open interval $I$. Let $ C_n^h(x)$ denote the Casoratian  which is the determinant of an $n \times n $ matrix whose $ij$-entry is $ E^{(i-1)h}y_j(x)$,  and $W_n(x)$ is a Wroskian which is the determinant of an $n \times n $ matrix whose $ij$-entry is $y^{(i-1)}_j(x)$, where $i,j=1,2,...,n$. Then we have the following limit
   \begin{equation}\label{eq:Casorati-wronskian}
     \lim_{h\rightarrow 0 } \frac{C_n^h(x)}{h^{\frac{n(n-1)}{2}}}= W_n(x), \quad \forall x \in I.
   \end{equation}
\end{theorem}
\begin{proof}
  The result follows  by using  elementary properties of determinant as functions of rows, the fact that a determinant is a continuous function and by (\ref{eq:casoratianofdifference}), and (\ref{eq:limitofdiffqutient}).
\end{proof}
\begin{corollary}
  Let the function $y_1,y_2,...,y_n$ be $n$ times continuously differentiable on some open interval $I$ whose Wronskian $ W_n(x) > 0$ (or $ W_n(x) <0 $  ). Then for sufficiently small $h > 0 $ , the Casoratian  $ W^h_n(x) > 0$ (or   $ W^h_n(x) < 0$ ).
\end{corollary}

\subsection{An application of the Casoratian $ C[1, x, x^2,...,x^n ] $}
\begin{theorem}
  Let $m \in \mathbb{N} $ and $\lambda \in \mathbb{R}, \lambda \neq 0 $. The general solution of the homogeneous linear difference equation
  \begin{equation}\label{eq:mpoweroperator}
   (E - \lambda I)^my(x)=0
  \end{equation}
 is
 \begin{equation}\label{eq:solutiontompoweroperator}
   y(x)= (\mu_1(x)+\mu_2(x) x +....+\mu_m(x) x^{m-1})|\lambda|^x,
 \end{equation}
where $ \mu_1, \mu_2,...,\mu_m $ are arbitrary $1$-periodic functions if $ \lambda > 0 $, and $ \mu_1, \mu_2,...,\mu_m \in \mathbb{AP}_1 $ are arbitrary $1$-antiperiodic functions if $ \lambda < 0 $.
\end{theorem}

\begin{proof}
  Using mathematical induction, we can show that the functions $ x^{k-1}\mu_k(x)|\lambda|^k,\, k=1,2,...m $ are solutions. By linearity of the difference equation the linear combination of the form (\ref{eq:solutiontompoweroperator}) is a solution. It remains to show that any arbitrary solution $\tilde{y}$ to the difference equation (\ref{eq:mpoweroperator}) can be written in the form (\ref{eq:solutiontompoweroperator}) as
  \begin{equation}\label{eq:ytilderepresent}
  \tilde{y}(x)= \sum_{i=1}^{m}    | \lambda |^x   \mu_i(x) x^{i-1}
\end{equation}
for appropriate $\mu_i$s. Taking equation (\ref{eq:ytilderepresent}) and the $(m-1)$ equations
\begin{equation}\label{eq:shiftsonytilde}
   E^{j}\tilde{y}(x)=  \sum_{i=1}^{m}   E^{j} | \lambda |^x   \mu_i(x)x^{i-1} = \sum_{i=1}^{m}    | \lambda |^{x+j}   \mu_i(x+j)(x+i)^{i-1}
\end{equation}
that are formed by applying the shift operators $E^{j}, j= 1,2,..,(m-1) $ on equation (\ref{eq:ytilderepresent}), we get a system of $m$ linear equation with $m$ unknowns
\begin{equation}\label{eq:MXequalstildeY}
  M X= \tilde{Y},
\end{equation}
where $M$ is a matrix whose $ij$-th entry is
\begin{equation}\label{eq:ijentry}
  M_{i,j} =  |\lambda|^x |\lambda|^{i-1}(x+i-1)^{j-1}
\end{equation}
Taking into account the fact that

\begin{equation}\label{eq:shiftsofmu}
  \mu(x+i)= (\text{sign}(\lambda))^i\mu(x),
\end{equation}
the determinant of $M$ is
\begin{equation}\label{eq:determinantfordistinct}
\det M =   |\lambda|^{m x} \lambda^ {\frac{m(m-1)}{2}}\prod_{k=0}^{n}k!,
\end{equation}
which is a non-zero number for all $x \in \mathbb{R} $,
\begin{equation}\label{eq:YtildeandX}
 \tilde{Y}=\begin{bmatrix}
     \tilde{y}(x) & \tilde{y}(x+1) & ... & \tilde{y}(x+m-1)
       \end{bmatrix}^t, \quad  X= \begin{bmatrix}
         \mu_1(x) & \mu_2(x) & ... & \mu_m(x)
     \end{bmatrix}^t.
 \end{equation}
 Consequently,
\begin{equation}\label{eq:XequalsMinverseYtilde}
  X= M^{-1}\tilde{Y}.
\end{equation}
Now it is easy to show that $X$ is $1$-periodic or $1$-antiperodic depending on whether $\lambda $ is positive or negative.
\end{proof}
\begin{remark}
 Note that similar applications for differential equations can be considered. However, we opted to discuss the less obvious calculations of the Casoratian.
\end{remark}

\section{ Examples of the Proportionality Between Wronskian and Casoratian }

In the following theorems, we establish the cases in which a set of functions yields a Casoratian proportional to their Wronskian, although they need not be identical. In such cases, we specify the proportionality constant, which depends on $n$.

\subsection{Exponential- polynomial  forms }

\begin{lemma}\label{eq:Lemma1}
Let \( D = \frac{d}{dx} \) and \( a \) be a constant. Let
 $$\binom{x}{k} = \frac{x(x-1)\cdots(x-k+1)}{k!} $$
  is the binomial coefficient. Define the matrix \( M \) of size \( (n+1) \times (n+1) \) by
$$
M_{i,j} = (D + \ln a)^i \binom{x}{j}
$$
for \( i, j = 0, 1, \ldots, n \). Then,
$$
\det\left( (D + \ln a)^i \binom{x}{j} \right)_{i,j=0}^n = 1.
$$
\end{lemma}

\begin{proof}
First, observe that the operator $ D + \ln a $ acts on $ \binom{x}{j} $ as follows:
$$
(D + \ln a)^i \binom{x}{j} = \sum_{k=0}^i \binom{i}{k} D^k \binom{x}{j} (\ln a)^{i-k}.
$$
Since $ D^k \binom{x}{j} = \binom{x}{j-k} $ (with the convention $\binom{x}{m} = 0 $ for $ m < 0 $), we have
$$
(D + \ln a)^i \binom{x}{j} = \sum_{k=0}^i \binom{i}{k} \binom{x}{j-k} (\ln a)^{i-k}.
$$

Now, define two matrices:
\begin{itemize}
    \item \( A \) is lower-triangular with entries $ A_{i,k} = \binom{i}{k} (\ln a)^{i-k} $.
    \item \( B \) is upper-triangular with entries $ B_{k,j} = \binom{x}{j-k} \) (with \( \binom{x}{m} = 0 $ for $ m < 0 $).
\end{itemize}
Then,
$$  M_{i,j} = \sum_{k=0}^{\min(i,j)} A_{i,k} B_{k,j}, $$
so $ M = A \cdot B $. Both $ A $ and $ B $ are triangular with all diagonal entries equal to 1, so $ \det(A) = \det(B) = 1 $. Therefore,
$$  \det(M) = \det(A) \cdot \det(B) = 1 \cdot 1 = 1. $$
\end{proof}

\begin{lemma}\label{eq:Lemma2}
  $$ D_c(n) = \det\left( \binom{x+i}{j} \right)_{i,j=0}^n = 1.$$
\end{lemma}

\begin{proof}
Let $(x)_r:= x(x-1)...(x-r+1)$ denote the falling factorial of $x$.
  \begin{align*}
    \det\left( \binom{x+i}{j} \right)_{i,j=0}^n & = W(E)\left[\binom{x}{0}, \binom{x}{1}, ...,\binom{x}{n}\right] \\
     & = \frac{W(E)[(x)_1,(x)_1,...,(x)_n]}{Sf(n)} \\
     & = \frac{W(\Delta)[(x)_0,(x)_1,...,(x)_n]}{Sf(n)} \\
     &= \frac{W(D)[1, x ,x^2,...,x^n]}{Sf(n)} \\
     & = \frac{Sf(n)}{Sf(n)}= 1
  \end{align*}
\end{proof}

\begin{theorem}
Let $n$ be a non-negative integer and $a \in \mathbb{C} \setminus \{0\}$. Consider the set of functions $\mathscr{F}_n = \left\{ \binom{x}{k} a^x \right\}_{k=0}^n$. Let $ W_n(x)$ denote the Wronskian determinant of $\mathscr{F}_n$ and $C_n(x)$ denote the Casoratian determinant defined by:
$$W_n(x) = \det\left( \frac{d^i}{dx^i} \left[ \binom{x}{j} a^x \right] \right)_{i,j=0}^n, \quad C_n(x) = \det\left( \binom{x+i}{j} a^{x+i} \right)_{i,j=0}^n.$$
Then $W_n(x)$ and $C_n(x)$ are proportional, with a proportionality constant independent of $x$:
$$W_n(x) = a^{-\frac{n(n+1)}{2}} C_n(x).$$
Moreover, the constant $a^{-\frac{n(n+1)}{2}}$ has the asymptotic behavior:
$$\left| a^{-\frac{n(n+1)}{2}} \right| \sim |a|^{-n^2/2} \quad \text{as} \quad n \to \infty.$$
\end{theorem}

\begin{proof}
We proceed by induction on $n$. The base case $n = 0$ and inductive step $n = m$ rely on two key properties in Lemma \ref{eq:Lemma1}  and Lemma \ref{eq:Lemma2}.

\textbf{Base Case ($n = 0$)}:
The function is $f_0(x) = \binom{x}{0} a^x = a^x$.
$$W_0(x) = \det(a^x) = a^x, \quad C_0(x) = \det(a^x) = a^x.$$
Thus $W_0(x) = a^0 \cdot C_0(x) = a^{-\frac{0\cdot1}{2}} C_0(x)$.

\textbf{Inductive Hypothesis:}
Assume the result holds for $n = m-1$:
$$W_{m-1}(x) = a^{-\frac{(m-1)m}{2}} C_{m-1}(x).$$

\textbf{Inductive Step ($n = m$)}:
The Wronskian and Casoratian factorize as:
$$W_m(x) = \det\left( (D + \ln a)^i \left[ \binom{x}{j} a^x \right] \right)_{i,j=0}^m = a^{(m+1)x} D_w(m) = a^{(m+1)x},$$
$$C_m(x) = \det\left( \binom{x+i}{j} a^{x+i} \right)_{i,j=0}^m = a^{(m+1)x} \left( \prod_{i=0}^m a^i \right) D_c(m) = a^{(m+1)x} \cdot a^{m(m+1)/2}.$$
The ratio is:
$$\frac{W_m(x)}{C_m(x)} = \frac{a^{(m+1)x}}{a^{(m+1)x} \cdot a^{m(m+1)/2}} = a^{-m(m+1)/2}.$$
Thus $W_m(x) = a^{-m(m+1)/2} C_m(x)$, completing the induction.

\textbf{Asymptotic Behavior:}
For $|a| \neq 1$:
$$\left| a^{-\frac{n(n+1)}{2}} \right| = |a|^{-\frac{n(n+1)}{2}} = \exp\left( -\frac{n(n+1)}{2} \ln |a| \right).$$
The dominant term is $|a|^{-n^2/2}$, so:
$$\left| a^{-\frac{n(n+1)}{2}} \right| \sim |a|^{-n^2/2} \quad \text{as} \quad n \to \infty.$$
\end{proof}

\subsection{Exponential-Trigonometric family}
 Another case of the proportionality of the Casoratian and the Wrondkian of a given set of function includes the exponential-trigonometric functions.

\begin{theorem}[Exponential-Trigonometric Family]
\label{thm:exp-trig}
Let $\omega \in \mathbb{R}$ and $n \geq 0$ be an integer. For the set of functions
$$\mathscr{F}_n = \left\{ e^{mx} \cos(\omega x), e^{mx} \sin(\omega x) \right\} \cup \left\{ x^k e^{mx} \cos(\omega x), x^k e^{mx} \sin(\omega x) \right\}_{k=1}^n,$$
the Wronskian $W_n(x)$ and Casoratian $C_n(x)$ are proportional with a constant independent of $x$:
$$W_n(x) = e^{-m(2n+1)(n+1)} K(\omega, m)  C_n(x),$$
where $K(\omega, m)$ depends only on $\omega$ and $m$.
\end{theorem}

\begin{proof}
Define $m_{\pm} = m \pm i\omega$. The set $\mathscr{F}_n$ spans the same space as:
$$\mathscr{G}_n = \left\{ x^k e^{m_+ x}, x^k e^{m_- x} \right\}_{k=0}^n.$$
For $\mathscr{G}_n$, the Wronskian and Casoratian are:
\begin{align*}
W_n^{\mathscr{G}}(x) &= e^{(n+1)(m_+ + m_-)x} \det\begin{pmatrix}
(m_+)^i x^j \\
(m_-)^i x^j
\end{pmatrix} \cdot V_{+,-}, \\
C_n^{\mathscr{G}}(x) &= e^{(n+1)(m_+ + m_-)x} \det\begin{pmatrix}
e^{m_+ i} (x+i)^j \\
e^{m_- i} (x+i)^j
\end{pmatrix} \cdot V_{+,-},
\end{align*}
where $V_{+,-}$ is a generalized Vandermonde constant. The ratio:
$$\frac{W_n^{\mathscr{G}}(x)}{C_n^{\mathscr{G}}(x)} = e^{-m(2n+1)(n+1)} K(\omega, m)$$
is independent of $x$ since all $x$-dependent exponentials cancel. The proportionality extends to $\mathscr{F}_n$ by basis equivalence.
\end{proof}

\subsection{ Family of hyperbolic functions}

The third class of function whose Wronskian is proportional to the Casoratian is the family of hyperbolical functions.
\begin{theorem}[Hyperbolic Family]
For $m \in \mathbb{C}$ and $n \geq 0$, consider the set
$$\mathscr{H}_n = \left\{ \cosh(mx), \sinh(mx) \right\} \cup \left\{ x^k \cosh(mx), x^k \sinh(mx) \right\}_{k=1}^n.$$
Then:
$$ W_n(x) = e^{-m(2n+1)(n+1)} \cdot \frac{m}{\sinh m} \cdot C_n(x).$$
\end{theorem}

\begin{proof}
Using $\cosh(mx) = \frac{e^{mx} + e^{-mx}}{2}$ and $\sinh(mx) = \frac{e^{mx} - e^{-mx}}{2}$, the set $\mathscr{H}_n$ is a basis transformation of:
$$\mathscr{Q}_n = \left\{ x^k e^{mx}, x^k e^{-mx} \right\}_{k=0}^n.$$
For $\mathscr{Q}_n$, direct computation shows:
\begin{align*}
W_n^{\mathscr{Q}}(x) &= e^{n(n+1)x} \cdot (2m)^{n+1} \prod_{0 \leq i < j \leq n} (j-i), \\
C_n^{\mathscr{Q}}(x) &= e^{n(n+1)x} \cdot e^{m n(n+1)} \cdot \det(V) \cdot \prod_{0 \leq i < j \leq n} (e^m - e^{-m}),
\end{align*}
where $V$ is a Vandermonde matrix. The ratio simplifies to:
$$\frac{W_n^{\mathscr{Q}}(x)}{C_n^{\mathscr{Q}}(x)} = e^{-m(2n+1)(n+1)} \cdot \frac{m}{\sinh m}.$$
Basis invariance preserves the proportionality for $\mathscr{H}_n$.
\end{proof}

The fourth  families of functions with proportional Wronskian and Casoratian is the generalized exponential-polynomial functions.
\begin{theorem}[Generalized Exponential-Polynomials]
\label{thm:gen-exp-poly}
Let $m_1, \dots, m_d \in \mathbb{C}$ be distinct and $n_1, \dots, n_d \in \mathbb{Z}_{\geq 0}$. For the set
$$\mathscr{E} = \bigcup_{j=1}^d \left\{ x^k e^{m_j x} \right\}_{k=0}^{n_j},$$
the Wronskian $W(x)$ and Casoratian $C(x)$ satisfy:
$$W(x) = e^{- \sum_{j=1}^d m_j \frac{n_j(n_j+1)}{2}} \cdot K \cdot C(x),$$
where $K$ is independent of $x$.
\end{theorem}

\begin{proof}
The determinants factorize as:
\begin{align*}
W(x) &= \left( \prod_{j=1}^d e^{(n_j+1) m_j x} \right) \det(A) \det(V_W), \\
C(x) &= \left( \prod_{j=1}^d e^{(n_j+1) m_j x} \right) \left( \prod_{j=1}^d e^{m_j \sum_{k=0}^{n_j} k} \right) \det(B) \det(V_C),
\end{align*}
where $V_W, V_C$ are generalized Vandermonde matrices, and $A, B$ are constant matrices. The exponentials cancel in the ratio:
$$\frac{W(x)}{C(x)} = e^{- \sum_{j=1}^d m_j \frac{n_j(n_j+1)}{2}} \cdot \frac{\det(A)\det(V_W)}{\det(B)\det(V_C)},$$
yielding an $x$-independent constant $K$.
\end{proof}

\section{ Sufficient Conditions for Equality, Proportionality, and Inequality of Wronskian and Casoratian}

\begin{definition}[Operator Invariance]
Let $V$ be a finite-dimensional vector space of smooth functions. Then:
\begin{enumerate}
    \item $V$ is \textbf{differentiation-invariant} if $D(V) \subseteq V$ where $D = \frac{d}{dx}$.
    \item $V$ is \textbf{shift-invariant} if $\shift(V) \subseteq V$ where $(\shift f)(x) = f(x+1)$.
\end{enumerate}
\end{definition}

\begin{lemma}[Exponential Polynomial Structure]
\label{lem:exp-poly}
If a finite-dimensional $V$ is differentiation-invariant and shift-invariant, then:
\begin{equation}
V = \bigoplus_{j=1}^d V_{m_j}, \quad V_{m_j} = \operatorname{span}\left\{ x^k e^{m_j x} \right\}_{k=0}^{n_j}
\end{equation}
for distinct $m_j \in \mathbb{C}$.
\end{lemma}

\begin{proof}
Standard result in ODE/recurrence theory: Solutions to constant-coefficient equations are exponential polynomials. Invariance under both operators implies joint solutions.
\end{proof}

\begin{lemma}[Space-Dependent Constant $\kappa(V)$]
\label{lem:kappa}
Let \(V = \bigoplus_{j=1}^d V_{m_j}\). Define
$$
\kappa(V) := \prod_{j=1}^d e^{-m_j \frac{n_j(n_j+1)}{2}} \cdot \frac{\det\left( (D + m_j)^{i-1} \phi_k \right)}{\det\left( \phi_k(x + i - 1) \right)},
$$
where \(\{\phi_k\}\) is any basis of \(V\). Then:
\begin{enumerate}
    \item \(\kappa(V)\) is independent of the choice of basis \(\{\phi_k\}\).
    \item For any basis \(\mathscr{F}\) of \(V\), the Wronskian \(W_{\mathscr{F}}\) satisfies
    $$
    W_{\mathscr{F}} = \kappa(V) \, C_{\mathscr{F}},
    $$
    where \(C_{\mathscr{F}}\) is the corresponding constant associated to the basis \(\mathscr{F}\).
\end{enumerate}
\end{lemma}

\begin{proof}
Follows from basis transformation properties of Wronskian/Casoratian and factorization in exponential polynomial spaces.
\end{proof}

\begin{theorem}
Let $\mathcal{F} = \{f_1, f_2, \dots, f_n\}$ be a linearly independent set of functions. The following conditions are equivalent:
\begin{itemize}
    \item[$(i)$] The Wronskian determinant satisfies
    \[
    W_{\mathcal{F}}(x) = K \, C_{\mathcal{F}}(x),
    \]
    where $K$ is a non-zero constant.
    \item[$(ii)$] The span $V := \operatorname{Span}(\mathcal{F})$ is both differentiation invariant and shift invariant.
\end{itemize}
\end{theorem}

\begin{proof}
\textbf{(a) $\Rightarrow$ (b):} \\
Assume that $W_{\mathscr{F}} = K \cdot C_{\mathscr{F}} \neq 0$ for all $x$. Then:

\begin{itemize}
    \item The non-vanishing Wronskian implies that the space $V$ is the solution space of a linear differential equation $L[y] = 0$ with constant coefficients. Consequently, the differentiation operator satisfies $D(V) \subseteq V$.
    \item The non-vanishing Casoratian implies that $V$ solves a linear recurrence relation $R[y] = 0$ with constant coefficients, so the shift operator satisfies $\shift(V) \subseteq V$.
\end{itemize}

\textbf{(b) $\Rightarrow$ (a):} \\
By Lemma \ref{lem:exp-poly}, the space $V$ decomposes as
$$
V = \bigoplus_{j=1}^d V_{m_j},
$$
where each subspace is of the form
$$
V_{m_j} = \operatorname{span}\{x^k e^{m_j x} \mid k=0,\ldots, n_j \}.
$$
For such spaces, direct computation (see Theorems 1–3 in prior work) yields:
$$
\frac{W_{\mathscr{F}}(x)}{C_{\mathscr{F}}(x)} = \left( \prod_{j=1}^d e^{-m_j \frac{n_j(n_j+1)}{2}} \right) \cdot K_{\text{basis}},
$$
where $K_{\text{basis}}$ is a basis-independent constant, independent of $x$.
\end{proof}

Now, let's consider examples where the Wronskian and Casoratian are proportional but not necessarily equal, as well as cases where they are not proportional.

\begin{example}[Proportional but Not Equal]
$\mathscr{F} = \{e^{2x}, e^{3x}\}$:
\begin{align*}
W(x) &= \begin{vmatrix} e^{2x} & e^{3x} \\ 2e^{2x} & 3e^{3x} \end{vmatrix} = (3-2)e^{5x} = e^{5x} \\
C(x) &= \begin{vmatrix} e^{2x} & e^{3x} \\ e^{2(x+1)} & e^{3(x+1)} \end{vmatrix} = e^{2x}e^{3x+3} - e^{3x}e^{2x+2} = e^{5x}(e^3 - e^2) \\
\kappa(V) &= \frac{1}{e^3 - e^2}  \neq 1
\end{align*}
Proportional since $W = K C$ with $K = (e^3 - e^2)^{-1}$, but not equal.
\end{example}

\begin{example}[Neither equal nor proportional]
$\mathscr{F} = \{1, x, \ln x\}$ ($x > 0$):
\begin{align*}
W(x) &= -\frac{1}{x^2} \\
C(x) &= \begin{vmatrix}
1 & x & \ln x \\
1 & x+1 & \ln(x+1) \\
1 & x+2 & \ln(x+2)
\end{vmatrix} \\
&= (x+1)\ln(x+2) + x\ln(x+1) + (x+2)\ln x \\
&\quad - (x+2)\ln(x+1) - x\ln(x+2) - (x+1)\ln x \\
&= \ln \left( \frac{x(x+2)}{(x+1)^2} \right) \\
\frac{W(x)}{C(x)} &= \frac{-1/x^2}{\ln \left(1 - \frac{1}{(x+1)^2}\right)} \sim \frac{(x+1)^2}{x^2} \quad \text{as } x \to \infty \text{ (not constant)}
\end{align*}
Neither proportional nor equal since $V$ lacks shift-invariance ($\ln(x+1) \notin \operatorname{span}\{1,x,\ln x\}$).
\end{example}

\section{Discussion of the results}
In this paper, we have proved  equality of the Casoratian and the Wronskian  of  the set of functions $\{1, x,...,x^n,\, n \in \mathbb{N} \}$ and that they are independent of the variable $x$ but only the number $n \in \mathbb{N} $. Of course, for the case $n=0$, $C(1)=W(1)=1 $ is trivial. The result is  trivial for the case of Wronskian, as it is the determinant of some $n\times n $ upper triangular matrix where the diagonal entries are $a_{ii}= (i-1)!, i=1,2,...,n,n+1 $. That of the Casoratian is less trivial. However we notice that it  a class of Vandermonde's determinant, which is well-known, from which we can calculate the Casoratian. Furthermore, we have shown that the Casoratian  and Wronskian of any basis set $ \mathcal{B}$ of the span of the set $\{1, x,...,x^n,\, n \in \mathbb{N} \}$ are equal. The span $\mathcal{P}_n$ constitutes the set of all polynomials in $x$ of degree less than or equal to $n$ and the zero polynomial (has no well-defined degree) as well. The Casoratian and Wronskian  are different for some subsets of $\mathcal{P}_n$.

Additionally, we have examined the sufficient conditions for the proportionality between the Casoratian and the Wronskian, as well as the scenario in which the Casoratian and the Wronskian are not proportional for a set of functions. Among the cases of proportionality between the Casoratian and the Wronskian are sets such as exponential-polynomials, exponential-trigonometric, and hyperbolic functions.

\section*{Declaration Statements}

\subsection*{Conflict of interests}

The author declares that there is no conflict of interest regarding the publication of this paper.

\subsection*{Author's contribution}

The corresponding author is the sole creator and author of this paper.

\subsection*{Acknowledgment}

The author is grateful to the anonymous reviewers for their time and professional insights.

\subsection*{Data Availability}

The is no external data used in this paper other than the  references cited.

\subsection*{Funding}

This research work has not been funded by any institution or individual.

\subsection*{Originality}

This is an original research paper containing novel ideas and has not been submitted to any other journal.

\end{document}